\documentclass[12pt]{article}%
\usepackage{amsmath}
\usepackage{amsfonts}
\usepackage{amssymb}
\usepackage{graphicx}%
\providecommand{\U}[1]{\protect\rule{.1in}{.1in}}
\newtheorem{theorem}{Theorem}

\newtheorem{algorithm}[theorem]{Algorithm}

\newtheorem{condition}[theorem]{Condition}

\newtheorem{corollary}[theorem]{Corollary}

\newtheorem{definition}[theorem]{Definition}

\newtheorem{lemma}[theorem]{Lemma}

\newtheorem{proposition}[theorem]{Proposition}

\newenvironment{proof}[1][Proof]{\noindent\textbf{#1.} }{\ \rule{0.5em}{0.5em}}
\begin{document}

\title{\textbf{Convergence and Perturbation Resilience of Dynamic String-Averaging
Projection Methods}}
\author{Yair Censor$^{1}$ and Alexander J. Zaslavski$^{2}\bigskip$\\$^{1}$Department of Mathematics, University of Haifa\\Mt.\ Carmel, Haifa 31905, Israel\\(yair@math.haifa.ac.il) \bigskip\\$^{2}$Department of Mathematics\\The Technion -- Israel Institute of Technology\\Technion City, Haifa 32000, Israel\\(ajzasl@techunix.technion.ac.il)}
\date{February 29, 2012. Revised: May 24, 2012.}
\maketitle

\begin{abstract}
We consider the convex feasibility problem (CFP) in Hilbert space and
concentrate on the study of string-averaging projection (SAP) methods for the
CFP, analyzing their convergence and their perturbation resilience. In the
past, SAP methods were formulated with a single predetermined set of strings
and a single predetermined set of weights. Here we extend the scope of the
family of SAP methods to allow iteration-index-dependent variable strings and
weights and term such methods dynamic string-averaging projection (DSAP)
methods. The bounded perturbation resilience of DSAP methods is relevant and
important for their possible use in the framework of the recently developed
superiorization heuristic methodology for constrained minimization problems.

\end{abstract}

\bigskip

\textbf{2010 Mathematics Subject Classification (MSC)}: 65K10, 90C25

\section{Introduction\label{sect:intro}}

In this paper we consider the \textit{convex feasibility problem} (CFP) in
Hilbert space $H.$ Let $C_{1},C_{2},\dots,C_{m}$ be nonempty closed convex
subsets of $H,$ where $m$ is a natural number, and define%
\begin{equation}
C:=\cap_{i=1}^{m}C_{i}. \label{eq:1.1}%
\end{equation}

Assuming consistency, i.e., that $C\not =\emptyset,$ the CFP requires to find
an element $x^{\ast}\in C.$ We concentrate on the study of
\textit{string-averaging projection} (SAP) \textit{methods} for the CFP and
analyze their convergence and their perturbation resilience. SAP methods were
first introduced in \cite{ceh01} and subsequently studied further in
\cite{bmr04,bdhk07,CS08,CS09,ct03,crombez}, see also \cite[Example 5.20]%
{BC11}. They were also employed in applications \cite{pen09,rhee03}.

The class of \textit{projection methods} is understood here as the class of
methods that have the property that they can reach an aim related to the
family of sets $\{C_{1},C_{2},\dots,C_{m}\},$ such as, but not only, solving
the CFP, or solving an optimization problem with these sets as constraints, by
performing projections (orthogonal, i.e., least Euclidean distance, or others)
onto the individual sets $C_{i}.$ The advantage of such methods occurs in
situations where projections onto the individual sets are computationally
simple to perform. Such methods have been in recent decades extensively
investigated mathematically and used experimentally with great success on some
huge and sparse real-world applications, consult, e.g.,
\cite{bb96,cccdh10,combettes01,combettes04} and the books
\cite{BC11,byrnebook,CEG12,CZ97,chinneck-book,ER11,galantai,GTH}.

Within the class of projection methods, SAP methods do not represent a single
algorithm but rather, what might be called, an \textit{algorithmic scheme,
}which means that by making a specific choice of strings and weights in SAP,
along with choices of other parameters in the scheme, a deterministic
algorithm for the problem at hand can be obtained.

In all these works, SAP methods were formulated with a single predetermined
set of strings and a single predetermined set of weights. Here we extend the
scope of the family of SAP methods to allow iteration-index-dependent variable
strings and weights. We term such SAP methods \textit{dynamic string-averaging
projection} (DSAP) \textit{methods.} This is reminiscence of the analogous
development of \textit{block-iterative projection} (BIP) \textit{methods} for
the CFP wherein iteration-index-dependent variable blocks and weights are
permitted \cite{ac89}. For such DSAP methods we prove here convergence and
bounded perturbation resilience.

The significance of DSAP in practice cannot be exaggerated. SAP methods, in
their earlier non-dynamic versions, have been applied to the important
real-world application of \textit{proton Computerized Tomography} (pCT), see,
e.g., \cite{pscr10, pen09}, which presents a computationally huge-size
problem. The efforts to use parallel computing for the application of SAP to
pCT is ongoing and will benefit from the DSAP. This is so because the
flexibility of varying string lengths, string members and weights dynamically
has direct bearing on load balancing between processors that run in parallel
and should be loaded in a way that will minimize idle time of processors that
await others to complete their jobs. Such experimental work will hopefully see
light elsewhere.

The so extended DSAP algorithmic scheme is presented in Section
\ref{sect:SAPv} and the convergence analysis of it is done in Section
\ref{sect:convergence}. In Section \ref{sect:perturbations} we quote the
definition of bounded perturbation resilience and prove that the DSAP with
iteration-index-dependent variable strings and weights is bounded perturbation
resilient. There we also comment about the importance and relevance of this
bounded perturbation resilience to the recently developed superiorization
heuristic methodology.

\section{The string-averaging projection method and the dynamic SAP with
variable strings and weights\label{sect:SAPv}}

Let $H$ be a real Hilbert space with inner product $\left\langle \cdot
,\cdot\right\rangle $ and induced norm $||\cdot||$. Originally,
string-averaging is more general than SAP because it can employ operators
other than projections and convex combinations. But, on the other hand, it is
formulated for fixed strings as follows. Let the \textit{string} $I_{t}$ be an
ordered subset of $\{1,2,\dots,m\}$ of the form%
\begin{equation}
I_{t}=(i_{1}^{t},i_{2}^{t},\dots,i_{m(t)}^{t}),\label{block}%
\end{equation}
with $m(t)$ the number of elements in $I_{t},$ for $t=1,2,\dots,M,$ where $M$
is the number of strings. Suppose that there is a set $S\subseteq H$ such that
there are operators $R_{1},R_{2},\dots,R_{m}$ mapping $S$ into $S$ and an
additional operator $R$ which maps $S^{M}$ into $S$.

\begin{algorithm}
\textbf{The string-averaging algorithmic scheme of \cite{ceh01}}

\textbf{Initialization}: $x^{0}\in S$ is arbitrary.

\textbf{Iterative Step}: given the current iterate $x^{k}$,

(i) calculate, for all $t=1,2,\dots,M,$%
\begin{equation}
T_{t}(x^{k})=R_{i_{m(t)}^{t}}\cdots R_{i_{2}^{t}}R_{i_{1}^{t}}(x^{k}),
\label{as1}%
\end{equation}
(ii) and then calculate%
\begin{equation}
x^{k+1}=R(T_{1}(x^{k}),T_{2}(x^{k}),\ldots,T_{M}(x^{k})). \label{as2}%
\end{equation}

\end{algorithm}

For every $t=1,2,\dots,M,$ this algorithmic scheme applies to $x^{k}$
successively the operators whose indices belong to the $t$-th string. This can
be done in parallel for all strings and then the operator $R$ maps all
end-points onto the next iterate $x^{k+1}$. This is indeed an algorithm
provided that the operators $\{R_{i}\}_{i=1}^{m}$ and $R$ all have algorithmic
implementations. In this framework we get a \textit{sequential algorithm} by
the choice $M=1$ and $I_{1}=(1,2,\dots,m)$ and a \textit{simultaneous
algorithm} by the choice $M=m$ and $I_{t}=(t),\ t=1,2,\dots,M.$

Now we proceed to construct our proposed DSAP method with variable strings and
weights. For each $x\in H$, nonempty set $E\subseteq H$ and $r>0$ define the
distance%
\begin{equation}
d(x,E)=\inf\{||x-y||\mid\;y\in E\}
\end{equation}
and the closed ball%
\begin{equation}
B(x,r)=\{y\in H\mid\;||x-y||\leq r\}.
\end{equation}

The following proposition and corollary are well-known.

\begin{proposition}
If $D$ be a nonempty closed convex subset of $H$ then for each $x\in H$ there
is a unique point $P_{D}(x)\in D,$ called \texttt{the projection of }%
$x$\texttt{ onto }$D,$ satisfying%
\begin{equation}
||x-P_{D}(x)||=\inf\{||x-y||\mid\ y\in D\}.
\end{equation}
Moreover,%
\begin{equation}
||P_{D}(x)-P_{D}(y)||\leq||x-y||\text{ for all }x,y\in H,
\end{equation}
and for each $x\in H$ and each $z\in D$,%
\begin{equation}
\left\langle z-P_{D}(x),x-P_{D}(x)\right\rangle \leq0.
\end{equation}

\end{proposition}

\begin{corollary}
\label{corr:1.1}If $D$ be a nonempty closed convex subset of $H$ then for each
$x\in H$ and each $z\in D$,%
\begin{equation}
||z-P_{D}(x)||^{2}+||x-P_{D}(x)||^{2}\leq||z-x||^{2}.
\end{equation}

\end{corollary}

For $i=1,2,\dots,m,$ we denote $P_{i}=P_{C_{i}}.$ An \textit{index vector} is
a vector $t=(t_{1},t_{2},\dots,t_{p})$ such that $t_{i}\in\{1,2,\dots,m\}$ for
all $i=1,2,\dots,p$. For a given index vector $t=(t_{1},t_{2},\dots,t_{q})$ we
denote its \textit{length} by $p(t)=q,$ and define the operator $P[t]$ as the
product of the individual projections onto the sets whose indices appear in
the index vector $t$, namely,%
\begin{equation}
\;P[t]:=P_{t_{q}}P_{t_{q-1}}\cdots P_{t_{1}},
\end{equation}
and call it a \textit{string operator}. A finite set $\Omega$ of index vectors
is called \textit{fit} if for each $i\in\{1,2,\dots,m\}$, there exists a
vector $t=(t_{1},t_{2},\dots,t_{p})\in\Omega$ such that $t_{s}=i$ for some
$s\in\{1,2,\dots,p\}$. For each index vector $t$ the string operator is
nonexpansive, since the individual projections are, i.e.,
\begin{equation}
||P[t](x)-P[t](y)||\leq||x-y||\text{ for all }x,y\in H,
\end{equation}
and also%
\begin{equation}
P[t](x)=x\text{ for all }x\in C.
\end{equation}

Denote by $\mathcal{M}$ the collection of all pairs $(\Omega,w)$, where
$\Omega$ is a fit finite set of index vectors and%
\begin{equation}
w:\Omega\rightarrow(0,\infty)\text{ is such that }\sum_{t\in\Omega}w(t)=1.
\label{eq:1.6}%
\end{equation}

A pair $(\Omega,w)\in\mathcal{M}$ and the function $w$ were called in
\cite{bdhk07} an \textit{amalgamator} and a \textit{fit weight function},
respectively. For any $(\Omega,w)\in\mathcal{M}$ define the convex combination
of the end-points of all strings defined by members of $\Omega$ by%
\begin{equation}
P_{\Omega,w}(x):=\sum_{t\in\Omega}w(t)P[t](x),\;x\in H. \label{eq:1.7}%
\end{equation}
It is easy to see that%
\begin{equation}
||P_{\Omega,w}(x)-P_{\Omega,w}(y)||\leq||x-y||\text{ for all }x,y\in H,
\label{eq:1.8}%
\end{equation}
and that%
\begin{equation}
P_{\Omega,w}(x)=x\text{ for all }x\in C. \label{eq:1.9}%
\end{equation}
We will make use of the following bounded regularity condition, see
\cite[Definition 5.1]{bb96}.

\begin{condition}
\label{cond:A}For each $\varepsilon>0$ and each $M>0$ there exists
$\delta=\delta(\varepsilon,M)>0$ such that for each $x\in B(0,M)$ satisfying
$d(x,C_{i})\leq\delta$, $i=1,2,\dots,m,$ the inequality $d(x,C)\leq
\varepsilon$ holds.
\end{condition}

The next proposition follows from \cite[Proposition 5.4(iii)]{bb96} but we
present its proof here for the reader's convenience.

\begin{proposition}
If the space $H$ is finite-dimensional then Condition \ref{cond:A} holds.
\end{proposition}

\begin{proof}
Assume to the contrary that there exist $\varepsilon>0$, $M>0$ and a sequence
$\{x^{k}\}_{k=0}^{\infty}\subset B(0,M)$ such that
\begin{equation}
\text{for each integer }k\geq1,\text{ }d(x^{k},C_{i})\leq1/k,\;i=1,2,\dots
,m\text{ and }d(x^{k},C)\geq\varepsilon. \label{eq:1.11}%
\end{equation}
We assume, without loss of generality, that there exists the limit%
\begin{equation}
\lim_{k\rightarrow\infty}x^{k}=\tilde{x}. \label{eq:1.12}%
\end{equation}
Then the closedness of $B(0,M)$ and (\ref{eq:1.11}) imply that%
\begin{equation}
\tilde{x}\in B(0,M)\cap C_{i},\;i=1,2,\dots,m.
\end{equation}
Together with (\ref{eq:1.1}) and (\ref{eq:1.12}) this implies that
$d(x^{k},C)<\varepsilon/2$ for all sufficiently large natural numbers $k$,
contradicting (\ref{eq:1.11}) and proving the proposition.
\end{proof}

In the sequel we will assume that Condition \ref{cond:A} holds. We fix a
number $\Delta\in(0,1/m)$ and an integer $\bar{q}\geq m$ and denote by
$\mathcal{M}_{\ast}\equiv\mathcal{M}_{\ast}(\Delta,\bar{q})$ the set of all
$(\Omega,w)\in\mathcal{M}$ such that the lengths of the strings are bounded
and the weights are all bounded away from zero, namely,%
\begin{equation}
\mathcal{M}_{\ast}:=\{(\Omega,w)\in\mathcal{M\mid}\text{ }p(t)\leq\bar
{q}\text{ and }w(t)\geq\Delta\text{ for all }t\in\Omega\}. \label{eq:11516}%
\end{equation}

The dynamic string-averaging projection (DSAP) method with variable strings
and variable weights is described by the following algorithm.

\begin{algorithm}
\label{alg:sap-v}$\left.  {}\right.  $\textbf{The DSAP method with variable
strings and variable weights}

\textbf{Initialization}: select an arbitrary $x^{0}\in H$,

\textbf{Iterative step}: given a current iteration vector $x^{k}$ pick a pair
$(\Omega_{k},w_{k})\in\mathcal{M}_{\ast}$ and calculate the next iteration
vector by%
\begin{equation}
x^{k+1}=P_{\Omega_{k},w_{k}}(x^{k})\text{.\label{eq:algv}}%
\end{equation}

\end{algorithm}

\section{Convergence analysis\label{sect:convergence}}

In this section we present our convergence analysis for the DSAP method with
variable strings and variable weights, Algorithm \ref{alg:sap-v}. The main
theorem is the following.

\begin{theorem}
\label{thm:1.1}Let the following assumptions hold:

(i) $M_{0}>0$ is such that%
\begin{equation}
B(0,M_{0})\cap C\not =\emptyset.
\end{equation}

(ii) $\varepsilon>0$, $M>0$ and $\delta>0$ are such that%
\begin{equation}
\text{ if }x\in B(0,2M_{0}+M)\text{ and }d(x,C_{i})\leq\delta,\;i=1,2,\dots
,m,\text{ then }d(x,C)\leq\varepsilon/4.
\end{equation}

(iii) $\gamma$ is a positive number that satisfies%
\begin{equation}
\bar{q}\gamma^{1/2}\leq\delta.
\end{equation}

(iv) $k_{0}$ is a natural number that satisfies%
\begin{equation}
k_{0}>(\gamma\Delta)^{-1}(M+M_{0})^{2}.
\end{equation}
Under these assumptions, if Condition \ref{cond:A} holds then any sequence
$\{x^{k}\}_{k=0}^{\infty}\subset H$, generated by Algorithm \ref{alg:sap-v}
with $||x^{0}||\leq M,$ converges in the norm of $H$, $\lim_{k\rightarrow
\infty}x^{k}\in C$ and%
\begin{equation}
||x^{k}-\lim_{s\rightarrow\infty}x^{s}||\leq\varepsilon\text{ for all integers
}k\geq k_{0}.
\end{equation}

\end{theorem}

We use the notation and the definitions from Section \ref{sect:SAPv} and prove
first the next two lemmas as tools for the proof of Theorem \ref{thm:1.1}.

\begin{lemma}
\label{lem:2.1}Let $t=(t_{1},t_{2},\dots,t_{p})$ be an index vector, $x\in H$
and $z\in C$. Then%
\begin{align}
||z-x||^{2}  &  \geq||z-P[t](x)||^{2}+||x-P_{t_{1}}(x)||^{2}\label{eq:2.1}\\
&  +\sum_{%
\begin{array}
[c]{c}%
{\scriptsize i=1}\\
{\scriptsize i<p}%
\end{array}
}^{p}\left\Vert P_{t_{i}+1}P_{t_{i}}\cdots P_{t_{1}}(x)-P_{t_{i}}\cdots
P_{t_{1}}(x)\right\Vert ^{2}.
\end{align}

\end{lemma}

\begin{proof}
By Corollary \ref{corr:1.1},%
\begin{equation}
||z-x||^{2}\geq||z-P_{t_{1}}(x)||^{2}+||x-P_{t_{1}}(x)||^{2} \label{eq:2.2}%
\end{equation}
and using the same corollary we also have, for each integer $i$ satisfying
$1\leq i<p$,%
\begin{align}
||z-P_{t_{i}}\cdots P_{t_{1}}x||^{2}  &  \geq||z-P_{t_{i}+1}P_{t_{i}}\cdots
P_{t_{1}}(x)||^{2}\nonumber\\
&  +||P_{t_{i}+1}P_{t_{i}}\cdots P_{t_{1}}x-P_{t_{i}}\cdots P_{t_{1}}%
(x)||^{2}. \label{eq:2.3}%
\end{align}
Combining (\ref{eq:2.2}) and (\ref{eq:2.3}) we obtain (\ref{eq:2.1}) and the
lemma is proved.
\end{proof}

For an index vector $t=(t_{1},t_{2},\dots,t_{p(t)})$ and a vector $x\in H$ let
us define the function%
\begin{equation}
\phi\lbrack t](x):=||x-P_{t_{1}}(x)||^{2}+\sum_{%
\begin{array}
[c]{c}%
{\scriptsize i=1}\\
{\scriptsize i<p}%
\end{array}
}^{p}\left\Vert P_{t_{i}+1}P_{t_{i}}\cdots P_{t_{1}}(x)-P_{t_{i}}\cdots
P_{t_{1}}(x)\right\Vert ^{2}. \label{eq:2.5}%
\end{equation}
Lemma \ref{lem:2.1} and (\ref{eq:2.5}) then imply that, for each $x\in H$ and
each $z\in C$,%
\begin{equation}
||z-x||^{2}\geq||z-P[t](x)||^{2}+\phi\lbrack t](x). \label{eq:2.6}%
\end{equation}

\begin{lemma}
\label{lem:2.2}Let $x\in H,$ $z\in C$ and $(\Omega,w)\in\mathcal{M}_{\ast}.$
Then%
\begin{equation}
||z-x||^{2}\geq||z-P_{\Omega,w}(x)||^{2}+\Delta\sum_{t\in\Omega}\phi\lbrack
t](x). \label{eq:2.8}%
\end{equation}

\end{lemma}

\begin{proof}
Since the function $u\rightarrow||u-z||^{2}$, $u\in H,$ is convex it follows
from (\ref{eq:1.6}), (\ref{eq:1.7}), (\ref{eq:2.5}), (\ref{eq:2.6}), and the
definition of $\mathcal{M}_{\ast}$ in (\ref{eq:11516}) that%
\begin{align}
||z-P_{\Omega,w}(x)||^{2}  &  =||z-\sum_{t\in\Omega}w(t)P[t](x)||^{2}\leq
\sum_{t\in\Omega}w(t)||z-P[t](x)||^{2}\nonumber\\
&  \leq\sum_{t\in\Omega}w(t)\left(  ||z-x||^{2}-\phi\lbrack t](x)\right)
=||z-x||^{2}-\sum_{t\in\Omega}w(t)\phi\lbrack t](x)\nonumber\\
&  \leq||z-x||^{2}-\Delta\sum_{t\in\Omega}\phi\lbrack t](x).
\end{align}
This implies (\ref{eq:2.8}) and completes the proof.
\end{proof}

Now we are ready to prove Theorem \ref{thm:1.1}.

\textbf{Proof of Theorem \ref{thm:1.1}.} We first wish to show that there is a
natural number $\ell\leq k_{0}$ such that%
\begin{equation}
\sum_{t\in\Omega_{\ell}}\phi\lbrack t](x^{\ell-1})\leq\gamma. \label{eq:2.9}%
\end{equation}
To this end we assume to the contrary, that%
\begin{equation}
\text{for all }k=1,2,\dots,k_{0},\text{ }\sum_{t\in\Omega_{k}}\phi\lbrack
t](x^{k-1})>\gamma, \label{eq2.10}%
\end{equation}
and take some $\theta\in B(0,M_{0})\cap C.$ By Lemma \ref{lem:2.2}, by the
fact that $(\Omega_{k},w_{k})\in\mathcal{M}_{\ast}$ for all $k\geq0,$ by
(\ref{eq:algv}) and by (\ref{eq2.10}), we have%
\begin{equation}
||\theta-x^{k-1}||^{2}\geq||\theta-x^{k}||^{2}+\Delta\sum_{t\in\Omega_{k}}%
\phi\lbrack t](x^{k-1})>||\theta-x^{k}||^{2}+\Delta\gamma.
\end{equation}
By the choice of $\theta$ and the fact that $||x^{0}||\leq M$ we obtain%
\begin{align}
(M+M_{0})^{2}  &  \geq||\theta-x^{0}||^{2}-||\theta-x^{k_{0}}||^{2}\nonumber\\
&  =\sum_{k=1}^{k_{0}}(||\theta-x^{k-1}||^{2}-||\theta-x^{k}||^{2})\geq
k_{0}\Delta\gamma
\end{align}
which implies%
\begin{equation}
k_{0}\leq(\Delta\gamma)^{-1}(M+M_{0})^{2}.
\end{equation}
This contradicts Assumption (iv) of the theorem thus showing that there exists
a natural number $\ell\leq k_{0}$ such that (\ref{eq:2.9}) holds.

From Lemma \ref{lem:2.2}, the choice of $\theta$, the fact that $(\Omega
_{k},w_{k})\in\mathcal{M}_{\ast}$ for all $k\geq0$, the iterative step
(\ref{eq:algv}) and $||x^{0}||\leq M$, we conclude that%
\begin{equation}
||\theta-x^{\ell-1}||^{2}\leq||\theta-x^{0}||^{2}\leq(M_{0}+M)^{2},
\end{equation}
which yields by the triangle inequality%
\begin{equation}
||x^{\ell-1}||\leq2M_{0}+M. \label{eq:2.12}%
\end{equation}
In order to use the last inequality to invoke Assumption (ii) of the theorem
we show next that%
\begin{equation}
d(x^{\ell-1},C_{i})\leq\delta,\;i=1,2,\dots,m.
\end{equation}
Assume that $s\in\{1,2,\dots,m\}$. Since the set $\Omega_{\ell}$ is fit there
is a $t=(t_{1},t_{2},\dots,t_{p(t)})\in\Omega_{\ell}$ such that%
\begin{equation}
s=t_{q}\text{ for some }q\in\{1,2,\dots,p(t)\}. \label{eq:2.14}%
\end{equation}
From (\ref{eq:2.9}) and the fact that $t\in\Omega_{\ell}$, we know that
$\phi\lbrack t](x^{\ell-1})\leq\gamma.$ Together with (\ref{eq:2.5}) this
implies that%
\begin{equation}
||x^{\ell-1}-P_{t_{1}}(x^{\ell-1})||\leq\gamma^{1/2}, \label{eq:2.15}%
\end{equation}
therefore, for any index $1\leq i\leq p(t)$ satisfying $i<p(t),$
\begin{equation}
||P_{t_{i+1}}P_{t_{i}}\cdots P_{t_{1}}(x^{\ell-1})-P_{t_{i}}\cdots P_{t_{1}%
}(x^{\ell-1})||\leq\gamma^{1/2}. \label{eq:2.16}%
\end{equation}
The fact that $(\Omega_{k},w_{k})\in\mathcal{M}_{\ast}$ for all $k\geq0$
guarantees that $p(t)\leq\bar{q},$ and the $\delta$ whose existence is
guaranteed by Condition \ref{cond:A}, Assumption (iii) of the theorem, along
with (\ref{eq:2.15}) and (\ref{eq:2.16}) imply that for any $i\in
\{1,\dots,p(t)\}$,%
\begin{equation}
||P_{t_{i}}\cdots P_{t_{1}}(x^{\ell-1})-x^{\ell-1}||\leq i\gamma^{1/2}\leq
p(t)\gamma^{1/2}\leq\bar{q}\gamma^{1/2}\leq\delta
\end{equation}
and thus that%
\begin{equation}
d(x^{\ell-1},C_{t_{i}})\leq\delta.
\end{equation}
Together with (\ref{eq:2.14}) this implies that%
\begin{equation}
d(x^{\ell-1},C_{s})\leq\delta. \label{eq:2.17}%
\end{equation}
Since (\ref{eq:2.17}) holds for any $s\in\{1,2,\dots,m\}$, we use
(\ref{eq:2.12}) and Assumption (ii) of the theorem to state that%
\begin{equation}
d(x^{\ell-1},C)\leq\varepsilon/4
\end{equation}
and that there is a $z\in H$ such that%
\begin{equation}
z\in C\text{ and}\;||x^{\ell-1}-z||<\varepsilon/3. \label{eq:2.18}%
\end{equation}
By (\ref{eq:1.8}), (\ref{eq:1.9}), (\ref{eq:2.18}), the fact that $(\Omega
_{k},w_{k})\in\mathcal{M}_{\ast}$ for all $k\geq0$ and the iterative step
(\ref{eq:algv}) we have%
\begin{equation}
||x^{k}-z||\leq||x^{\ell-1}-z||<\varepsilon/3\text{ for all integers }%
k\geq\ell-1. \label{eq:2.19}%
\end{equation}
This implies that for all integers $k_{1},k_{2}\geq\ell-1$ it is true that
$||x^{k_{1}}-x^{k_{2}}||<\varepsilon.$ Since $\varepsilon>0$ is arbitrary it
follows that $\{x^{k}\}_{k=0}^{\infty}$ is a Cauchy sequence and that the
limit $\lim_{k\rightarrow\infty}x^{k}$ in the norm exists. By (\ref{eq:2.19})%
\begin{equation}
||z-\lim_{k\rightarrow\infty}x^{k}||\leq\varepsilon/3.
\end{equation}
Since $\varepsilon>0$ is arbitrary it follows from (\ref{eq:2.18}) that
$\lim_{k\rightarrow\infty}x^{k}\in C$. By (\ref{eq:2.19}), since $\ell\leq
k_{0},$ and using (\ref{eq2.10}) for all integers $k\geq k_{0}$ we may write%
\begin{equation}
||x^{k}-\lim_{s\rightarrow\infty}x^{s}||\leq||x^{k}-z||+||z-\lim
_{s\rightarrow\infty}x^{s}||\leq\varepsilon/3+\varepsilon/3
\end{equation}
which completes the proof of Theorem \ref{thm:1.1}. $\blacksquare$

\section{Perturbation resilience of dynamic string-averaging with variable
strings and weights\label{sect:perturbations}}

In this section we prove the bounded perturbation resilience of the DSAP
method with variable strings and weights. We use the notations and the
definitions from the previous sections. The next definition was originally
given with a finite-dimensional Euclidean space $R^{J}$ instead of the Hilbert
space $H$ that we inserted into it below.

\begin{definition}
\label{def:resilient}\cite[Definition 1]{cdh10} Given a problem $T,$ an
algorithmic operator $\mathcal{A}:H\rightarrow H$ is said to be
\texttt{bounded perturbations resilient}\emph{ }if the following is true: if
the sequence $\{x^{k}\}_{k=0}^{\infty},$ generated by $x^{k+1}=\mathcal{A}%
(x^{k}),$ for all $k\geq0,$ converges to a solution of $T$ for all $x^{0}\in
H$, then any sequence $\{y^{k}\}_{k=0}^{\infty}$ of points in $H$ generated by
$y^{k+1}=\mathcal{A}(y^{k}+\beta_{k}v^{k}),$ for all $k\geq0,$ also converges
to a solution of $T$ provided that, for all $k\geq0$, $\beta_{k}v^{k}$ are
\texttt{bounded perturbations}, meaning that $\beta_{k}\geq0$ for all $k\geq0$
such that ${\displaystyle\sum\limits_{k=0}^{\infty}}\beta_{k}\,<\infty$ and
the sequence $\{v^{k}\}_{k=0}^{\infty}$ is bounded.
\end{definition}

We will make use of the following theorem that was proved in \cite[Theorem
3.2]{brz08}.

\begin{theorem}
\label{thm:brz08}Let $(Y,\rho)$ be a complete metric space, let $F\subset Y$
be a nonempty closed set, and let $T_{i}:Y\rightarrow Y$, $i=1,2,\dots,$
satisfy%
\begin{equation}
\rho(T_{i}(x),T_{i}(y))\leq\rho(x,y)\text{ for all }x,y\in Y\text{ and all
integers }i\geq1,
\end{equation}
and%
\begin{equation}
T_{i}(z)=z\text{ for each }z\in F\text{ and each integer }i\geq1.
\end{equation}
Assume that for each $x\in Y$ and integer $q\geq1$, the sequence
$\{T_{n}\cdots T_{q}(x)\}_{n=q}^{\infty}$ converges to an element of $F$. Let
$x^{0}\in Y$, $\{\gamma_{n}\}_{n=1}^{\infty}\subset(0,\infty)$, $\sum
_{n=1}^{\infty}\gamma_{n}<\infty$, $\{x^{n}\}_{n=0}^{\infty}\subset Y$, and
suppose that for all integers $n\geq0$,%
\begin{equation}
\rho(x^{n+1},T_{n+1}(x^{n}))\leq\gamma_{n+1}.
\end{equation}
Then the sequence $\{x^{n}\}_{n=0}^{\infty}$ converges to an element of $F$.
\end{theorem}

The next theorem establishes the bounded perturbations resilience of DSAP.

\begin{theorem}
\label{thm:4.1}Let $C_{1},C_{2},\dots,C_{m}$ be nonempty closed convex subsets
of $H,$ where $m$ is a natural number, $C:=\cap_{i=1}^{m}C_{i}\not =\emptyset
$, let $\{\beta_{k}\}_{k=0}^{\infty}$ be a sequence of nonnegative numbers
such that $\sum_{k=0}^{\infty}\beta_{k}<\infty$, let $\{v^{k}\}_{k=0}^{\infty
}\subset H$ be a norm bounded sequence, let $\{(\Omega_{k},w_{k}%
)\}_{k=1}^{\infty}\subset\mathcal{M}_{\ast},$ for all $k\geq0,$ and let
$x^{0}\in H.$ Then any sequence $\{x_{k}\}_{k=0}^{\infty},$ generated by
Algorithm \ref{alg:sap-v} in which (\ref{eq:algv}) is replaced by%
\begin{equation}
x^{k+1}=P_{\Omega_{k},w_{k}}(x^{k}+\beta_{k}v^{k}%
)\text{,\label{eq:algv-perturbed}}%
\end{equation}
converges in the norm of $H$ and its limit belongs to $C$.
\end{theorem}

\begin{proof}
The proof follows from Theorem \ref{thm:1.1} and from Theorem \ref{thm:brz08}.
\end{proof}

We conclude with a comment about the importance and relevance of this bounded
perturbation resilience to the recently developed superiorization methodology.
The superiorization methodology was first proposed (although without using
this term) in \cite{bdhk07} and subsequently investigated and developed
further in \cite{cdh10,dhc09,hd08,hgdc12,ndh12,pscr10}. For the case of a
minimization problem of an objective function $\phi$ over a family of
constraints $\{C_{i}\}_{i=1}^{m},$ where each set $C_{i}$ is a nonempty closed
convex subset of $R^{n}$ it works as follows. It applies to $C=\cap_{i=1}%
^{m}C_{i}$ a feasibility-seeking \textit{projection method} capable of using
projections onto the individual sets $C_{i}$ in order to generate a sequence
$\{x^{k}\}_{k=0}^{\infty}$ that converges to a point $x^{\ast}\in C.$ It
applies to $C$ only such a feasibility-seeking projection method which is
bounded perturbation resilient. Doing so, the superiorization method exploits
this perturbation resilience to perform objective function reduction steps by
doing negative subgradient moves with certain step sizes. Thus, in
superiorization the feasibility-seeking algorithm leads the overall process
and uses permissible perturbations, that do not spoil the feasibility-seeking,
to periodically jump out from the overall process to do the subgradient
objective function reduction step.

This has a great potential computational advantage and poses also interesting
mathematical questions. It has been shown to be advantageous in some
real-world problems in image reconstruction from projections, see the above
mentioned references. Theorem \ref{thm:4.1} here, which established the
bounded perturbations resilience of DSAP methods, makes it now possible to
use, when the need arises, also DSAP methods within the framework of the
superiorization heuristic methodology.\medskip

\textbf{Acknowledgments}. We gratefully acknowledge enlightening discussions
with Gabor Herman on the topics discussed in this paper. The work of the first
author was supported by the United States-Israel Binational Science Foundation
(BSF) Grant number 200912 and US Department of Army Award number W81XWH-10-1-0170.


\begin{thebibliography}{99}                                                                                               %


\bibitem {ac89}R. Aharoni and Y. Censor, Block-iterative projection methods
for parallel computation of solutions to convex feasibility problems.
\textit{Linear Algebra and its Applications} \textbf{120 }(1989), 165--175.

\bibitem {bb96}H.H. Bauschke and J.M. Borwein, On projection algorithms for
solving convex feasibility problems,\ \textit{SIAM Review} \textbf{38} (1996), 367--426.

\bibitem {BC11}H.H. Bauschke and P.L. Combettes, \textit{Convex Analysis and
Monotone Operator Theory in Hilbert Spaces}, Springer, New York, NY, USA, 2011.

\bibitem {bmr04}H.H. Bauschke, E. Matou\v{s}kov\'{a} and S. Reich, Projection
and proximal point methods: convergence results and counterexamples,
\textit{Nonlinear Analysis: Theory, Methods \& Applications} \textbf{56}
(2004), 715--738.

\bibitem {bdhk07}D. Butnariu, R. Davidi, G.T. Herman, and I.G. Kazantsev,
Stable convergence behavior under summable perturbations of a class of
projection methods for convex feasibility and optimization problems,
\textit{IEEE Journal of Selected Topics in Signal Processing} \textbf{1}
(2007), 540--547.

\bibitem {brz08}D. Butnariu, S. Reich and A.J. Zaslavski, Stable convergence
theorems for infinite products and powers of nonexpansive mappings,
\textit{Numerical Functional Analysis and Optimization} \textbf{29} (2008), 304--323.

\bibitem {byrnebook}C.L. Byrne,\textit{ Applied Iterative Methods, }AK Peters,
Wellsely, MA, USA, 2008.

\bibitem {CEG12}A. Cegielski, \textit{Iterative Methods for Fixed Point
Problems in Hilbert Spaces}, Lecture Notes in Mathematics, Springer, to appear.

\bibitem {cccdh10}Y. Censor, W. Chen, P.L. Combettes, R. Davidi and G.T.
Herman, On the effectiveness of projection methods for convex feasibility
problems with linear inequality constraints, \textit{Computational
Optimization and Applications} \textbf{51} (2012), 1065--1088.

\bibitem {cdh10}Y. Censor, R. Davidi and G.T. Herman, Perturbation resilience
and superiorization of iterative algorithms, \textit{Inverse Problems}
\textbf{26} (2010), 065008 (12pp).

\bibitem {ceh01}Y. Censor, T. Elfving and G.T. Herman, Averaging strings of
sequential iterations for convex feasibility problems. In: D. Butnariu, Y.
Censor and S. Reich (editors), \textit{Inherently Parallel Algorithms in
Feasibility and Optimization and Their Applications}, Elsevier Science
Publishers, Amsterdam, 2001, pp. 101--114.

\bibitem {CS08}Y. Censor and A. Segal, On the string averaging method for
sparse common fixed point problems, \textit{International Transactions in
Operational Research} \textbf{16 }(2009), 481--494.

\bibitem {CS09}Y. Censor and A. Segal, On string-averaging for sparse problems
and on the split common fixed point problem, \textit{Contemporary Mathematics}
\textbf{513} (2010), 125--142.

\bibitem {ct03}Y. Censor and E. Tom, Convergence of string-averaging
projection schemes for inconsistent convex feasibility problems,
\textit{Optimization Methods and Software} \textbf{18 }(2003), 543--554.

\bibitem {CZ97}Y. Censor and S.A. Zenios, \textit{Parallel Optimization:
Theory, Algorithms, and Applications}, Oxford University Press, New York, NY,
USA, 1997.

\bibitem {chinneck-book}J.W. Chinneck, \textit{Feasibility and Infeasibility
in Optimization: Algorithms and Computational Methods,} Springer, New York,
NY, USA, 2007.

\bibitem {combettes01}P.L. Combettes, Quasi-Fej\'{e}rian analysis of some
optimization algorithms, in: \textit{Inherently Parallel Algorithms in
Feasibility and Optimization and Their Applications}, (D. Butnariu, Y. Censor,
and S. Reich, Eds.), pp. 115--152, Elsevier, New York, NY, USA, 2001.

\bibitem {combettes04}P.L. Combettes, Solving monotone inclusions via
compositions of nonexpansive averaged operators, \textit{Optimization}
\textbf{53 }(2004), 475--504.

\bibitem {crombez}G. Crombez, Finding common fixed points of strict
paracontractions by averaging strings of sequential iterations,
\textit{Journal of Nonlinear and Convex Analysis} \textbf{3 }(2002), 345--351.

\bibitem {dhc09}R. Davidi, G.T. Herman, and Y. Censor, Perturbation-resilient
block-iterative projection methods with application to image reconstruction
from projections, \textit{International Transactions in Operational Research}
\textbf{16 }(2009), 505--524.

\bibitem {ER11}R. Escalante and M. Raydan, \textit{Alternating Projection
Methods}, Society for Industrial and Applied Mathematics (SIAM), Philadelphia,
PA, USA, 2011.

\bibitem {galantai}A. Gal\'{a}ntai, \textit{Projectors and Projection
Methods}, Kluwer Academic Publishers, Dordrecht, The Netherlands, 2004.

\bibitem {GTH}G.T. Herman, \textit{Fundamentals of Computerized Tomography:
Image Reconstruction from Projections}, Springer-Verlag, London, UK, 2nd
Edition, 2009.

\bibitem {hd08}G.T. Herman and R. Davidi, Image reconstruction from a small
number of projections, \textit{Inverse Problems} \textbf{24} (2008), 045011 (17pp).

\bibitem {hgdc12}G.T. Herman, E. Gardu\~{n}o, R. Davidi and Y. Censor,
Superiorization: An optimization heuristic with application to tomography,
Technical Report, January 12, 2012.

\bibitem {ndh12}T. Nikazad, R. Davidi and G.T. Herman, Accelerated
perturbation-resilient block-iterative projection methods with application to
image reconstruction, \textit{Inverse Problems} \textbf{28} (2012), 035005
(19pp).\qquad

\bibitem {pen09}S.N. Penfold, R.W. Schulte, Y. Censor, V. Bashkirov, S.
McAllister, K.E. Schubert and A.B. Rosenfeld, Block-iterative and
string-averaging projection algorithms in proton computed tomography image
reconstruction. In: Y. Censor, M. Jiang and G. Wang (editors),
\textit{Biomedical Mathematics: Promising Directions in Imaging, Therapy
Planning and Inverse Problems}, Medical Physics Publishing, Madison, WI, USA,
2010, pp. 347--367.

\bibitem {pscr10}S.N. Penfold, R.W. Schulte, Y. Censor and A.B. Rosenfeld,
Total variation superiorization schemes in proton computed tomography image
reconstruction, \textit{Medical Physics} \textbf{37 }(2010), 5887--5895.

\bibitem {rhee03}H. Rhee, An application of the string averaging method to
one-sided best simultaneous approximation, \textit{J. Korea Soc. Math. Educ.
Ser. B: Pure Appl. Math.} \textbf{10 }(2003), 49--56.
\end{thebibliography}
\end{document}